\documentclass [reqno,10pt,oneside]{amsart}

\usepackage{amsthm}
\usepackage{amscd}
\usepackage{amsmath,amssymb}
\usepackage{algorithm}
\usepackage[noend]{algorithmic}

\theoremstyle {plain}
\newtheorem {thm}{Theorem}[section]

\newtheorem {lem}[thm]{Lemma}
\newtheorem {cor}[thm]{Corollary}
\theoremstyle {definition}
\newtheorem {defn}[thm]{Definition}

\theoremstyle {remark}
\newtheorem {rem}[thm]{Remark}

\newtheorem {exmp}[thm]{Example}

\DeclareMathOperator{\LM}{LM}
\DeclareMathOperator{\LT}{LT}
\DeclareMathOperator{\LC}{LC}
\DeclareMathOperator{\LE}{LE}

\newcommand{\C}{{\mathbb C}}

\newcommand{\CF}{{\mathcal F}}

\newcommand{\Q}{{\mathbb Q}}

\newcommand{\gen}[1]{\langle #1 \rangle}

\newcommand{\dimK}[1]{\dim_K\big( #1 \big)}
\newcommand{\dimQ}[1]{\dim_\Q\left( #1 \right)}
\newcommand{\singular}{{\sc Singular }}

\begin{document}

\bibliographystyle{alpha}

\title{Solving via Modular Methods}

\author{Deeba Afzal}
\address{Deeba Afzal\\ Abdus Salam School of Mathematical Sciences\\ GC University\\ 
Lahore\\ 68-B\\ New Muslim Town\\ Lahore 54600\\ Pakistan}
\email{deebafzal@gmail.com}


\author{Faira Kanwal}
\address{Faira Kanwal\\ Abdus Salam School of Mathematical Sciences\\ GC University\\ 
Lahore\\ 68-B\\ New Muslim Town\\ Lahore 54600\\ Pakistan}
\email{fairakanwaljanjua@gmail.com}

\author{Gerhard Pfister}
\address{Gerhard Pfister\\ Department of Mathematics\\ University of Kaiserslautern\\
Erwin-Schr\"odinger-Str.\\ 67663 Kaiserslautern\\ Germany}
\email{pfister@mathematik.uni-kl.de}

\author{Stefan Steidel}
\address{Stefan Steidel\\ Department of Mathematical Methods in Dynamics and Durability\\
 Fraunhofer Institute for Industrial Mathematics ITWM\\ 
 Fraunhofer-Platz 1\\ 67663 Kaiserslautern\\ Germany}
\email{stefan.steidel@itwm.fraunhofer.de} 


\date{\today}

\maketitle

\begin{abstract}
In this article we present a parallel modular algorithm to compute all solutions with multiplicities 
of a given zero-dimensional polynomial system of equations over the rationals.
In fact, we compute a triangular decomposition using M\"oller's algorithm (cf.\ \cite{M93}) of the 
corresponding ideal in the polynomial ring over the rationals using modular methods, and then 
apply a solver for univariate polynomials.
\end{abstract}

\section{Introduction} \label{secIntro}

One possible approach\footnote{In \singular (cf.\ \cite{DGPS12}) this approach is implemented
in the library \texttt{solve.lib} on the basis of an univariate \emph{Laguerre solver} (cf.\ \cite[\S 8.9-8.13]{RR78}).} 
to find the solutions of a zero-dimensional system of multivariate polynomials is the triangular decomposition
of the corresponding ideal since triangular systems of polynomials can be solved using a univariate
solver recursively. 
There are already several results in this direction including an implementation in \emph{Maple}
(cf.\ \cite{AM99}, \cite{LMX05}, \cite{M00}). The technique to compute triangular sets has been refined (cf.\ 
\cite{DMSWX05}, \cite{CLMPX07}, \cite{LMX06}), modularized and parallelized (cf.\ \cite{MX07}, \cite{LM07}). 
We report about a modular and parallel version of the solver in \textsc{Singular} (cf.\ \cite{DGPS12}) and 
focus mainly on a probabilistic algorithm to compute the solutions of a polynomial system with multiplicities.

\section{Preliminary technicalities} \label{secPreTech}

We recall the definition and some properties of a triangular decomposition.
For details we refer to \cite{GP07}, \cite{M93} and \cite{M97}.

Let $K$ be a field, $X=\{x_1,\ldots,x_n\}$ a set of variables, $I \subseteq K[X]$ a zero-dimensional ideal, 
and we fix $>$ to be the lexicographical ordering induced by $x_1>\ldots>x_n$.  If $f \in K[X]$ is
a polynomial, then we denote by $\LE(f)$ the leading exponent of $f$, by $\LC(f)$ the leading coefficient 
of $f$, by $\LM(f)$ the leading monomial of $f$, and by $\LT(f) = \LC(f) \cdot \LM(f)$ the leading term of $f$.

\begin{defn}
A set of polynomials $F = \{f_1,\ldots,f_n\} \subseteq K[X]$ is called a \emph{triangular set} if
$\LT(f_i) = x_{n-i+1}^{\alpha_i}$ for some $\alpha_i > 0$ and each $i=1,\ldots,n$.

A list $\CF = (F_1,\ldots,F_s)$ of triangular sets is called a \emph{triangular decomposition} of a zero-dimensional
ideal $I \subseteq K[X]$ if $\sqrt I = \sqrt{\gen{F_1}} \cap \ldots \cap \sqrt{\gen{F_s}}$.
\end{defn}

\pagebreak

\begin{rem} \
\begin{enumerate}
\item A triangular set is a Gr\"obner basis.
\item A minimal Gr\"obner basis of a maximal ideal is a triangular set, and the primary decomposition of
          $\sqrt I$ is a triangular decomposition of $I$.
\end{enumerate}
\end{rem}

The following two lemmata are the basis for the algorithm by M\"oller (cf.\ \cite{M93}) which avoids primary 
decomposition.

\begin{lem} \label{lemMoeller1}
Let $G = \{g_1,\ldots,g_m\}$ be a reduced Gr\"obner basis of the zero-dimen\-sional ideal $I \subseteq K[X]$
such that $\LM(g_1)<\ldots<\LM(g_m)$. Moreover, for $i = 1,\ldots,m$, let $\alpha_i = \LE(g_i)$ in 
$\left(K[x_2,\ldots,x_n]\right)[x_1]$, i.e.\ $g_i=\sum_{j=0}^{\alpha_i} g_{ij}^\prime x_1^j$ for suitable
$g_{ij}^\prime \in K[x_2,\ldots,x_n]$. 
Then $G^\prime = \{g_{1\alpha_1}^\prime,\ldots,g_{m-1\alpha_{m-1}}^\prime\}$ is a Gr\"obner basis of 
$\gen{g_1,\ldots,g_{m-1}}:g_m$, and we have $\gen{G^\prime,g_m} = \gen{G^\prime,G}$.
\end{lem}

\begin{proof}
The proof can be found in \cite[Lemma 7]{M93}.
\end{proof}

\begin{lem} \label{lemMoeller2}
Let $I \subseteq K[X]$ be  a zero-dimensional ideal, and $h \in K[X]$. Then the following hold.
\begin{enumerate}
\item $\sqrt I = \sqrt{\gen{I,h}} \cap \sqrt{I:h}$.
\item $\dimK{K[X]/I} = \dimK{K[X]/\gen{I,h}} + \dimK{K[X]/(I:h)}$.
\end{enumerate}
\end{lem}

\begin{proof} \
\begin{enumerate}
\item The proof is an exercise in \cite[Exercise 4.5.3]{GP07}.
\item We consider two exact sequences. The first one $$0 \longrightarrow (I:h)/I \longrightarrow K[X]/I
          \stackrel{\cdot h}{\longrightarrow} K[X]/I \longrightarrow K[X]/\gen{I,h} \longrightarrow 0$$ yields $\dimK{(I:h)/I}
          = \dimK{K[X]/\gen{I,h}}$, and the second one $$0 \longrightarrow (I:h)/I \longrightarrow K[X]/I \longrightarrow
          K[X]/(I:h) \longrightarrow 0$$ yields $\dimK{K[X]/I} = \dimK{(I:h)/I}+\dimK{K[X]/(I:h)}$. Summarized, we obtain
          $$\dimK{K[X]/I} = \dimK{K[X]/\gen{I,h}} + \dimK{K[X]/(I:h)}.$$
\end{enumerate}
\end{proof}

Consequently, for $h_1,\ldots,h_l \in K[X]$ and $h_{l+1} = 1$, we are able to apply Lemma \ref{lemMoeller2} 
inductively and obtain
\begin{align*}
\sqrt I & = \sqrt{\gen{I,h_1}} \cap \sqrt{I:h_1} \\
           & = \sqrt{\gen{I,h_1,h_2}} \cap \sqrt{\gen{I,h_1}:h_2} \cap \sqrt{I:h_1} \\
           & = \ldots \\
           & = \sqrt{\gen{I,h_1,\ldots,h_l}} \cap \left( \bigcap_{i=1}^{l+1} \sqrt{\gen{I,h_1,\ldots,h_{i-1}}:h_i} \right).
\end{align*}
Together with Lemma \ref{lemMoeller1} we conclude the following corollary.

\begin{cor} \label{corMoeller}
With the assumption of Lemma \ref{lemMoeller1}, let $G^\prime \smallsetminus G = \{h_1,\ldots,h_l\}$. Then the
following hold.
\begin{enumerate}
\item If $G^\prime \smallsetminus G \neq \emptyset$, then $I \subsetneq \gen{G,h_1}$ and $I \subsetneq 
          \gen G:h_1$.
\item $\sqrt I = \sqrt{\gen{G^\prime,g_m}} \cap \big(\bigcap_{i=1}^{l+1} \sqrt{\gen{G,h_1,\ldots,h_{i-1}}:h_i}\big)$
          and, in addition, $I \subseteq \gen{G^\prime,g_m} \cap \big(\gen{G,h_1,\ldots,h_{i-1}}:h_i\big)$.
\item $\dimK{K[X]/I} = \sum_{i=1}^{l+1} \dimK{K[X]/(\gen{G,h_1,\ldots,h_{i-1}}:h_i)}$.
\end{enumerate}
\end{cor}

With regard to Corollary \ref{corMoeller}(2), especially $\gen{G,h_1,\ldots,h_l}=\gen{G^\prime,g_m}$ with 
$G^\prime \subseteq K[x_2,\ldots,x_n]$ is predestined for induction since $\sqrt{\gen{G^\prime,g_m}} = 
\sqrt{\gen{F_1^\prime,g_m}} \cap \ldots \cap \sqrt{\gen{F_s^\prime,g_m}}$ if $\CF^\prime = (F_1^\prime,\ldots,
F_s^\prime)$ is a triangular decomposition of $G^\prime$.
Referring to Corollary \ref{corMoeller}(3), the triangular decomposition obtained by iterating the approach 
of Corollary \ref{corMoeller} respects the multiplicities of the zeros of $I$. Therefore the zero-sets of different 
triangular sets are in general not disjoint as the following example shows.

\begin{exmp}
Let $G = \{x_2^{10},x_1x_2^3+x_2^5,x_1^{11}\} \subseteq \Q[x_1,x_2]$. Then we obtain $G^\prime=
\{x_2^{10},x_2^3\}$, $G^\prime \smallsetminus G = \{x_2^3\}$, and the triangular decomposition 
$\CF=(F_1,F_2)$ of $\gen G$ with $F_1 = \gen G :x_2^3 = \gen{x_2^7,x_1+x_2^2}$ and $F_2 = 
\gen{G,x_2^3} = \gen{x_2^3,x_1^{11}}$. Moreover, it holds $\dimQ{\Q[x_1,x_2]/\gen G} = 40$, 
$\dimQ{\Q[x_1,x_2]/(\gen G : x_2^3)}=7$, and $\dimQ{\Q[x_1,x_2]/\gen{G,x_2^3}} = 33$.
Note that $\gen G \subsetneq F_1 \cap F_2$.
\end{exmp}

Algorithm \ref{algTriangDecomp} shows the algorithm by M\"oller to compute the triangular decomposition of
a zero-dimensional ideal.\footnote{The corresponding procedure is implemented in \singular in the library
\texttt{triang.lib}.}

\begin{algorithm}[H]                    
\caption{Triangular Decomposition (\texttt{triangM})}         
\label{algTriangDecomp}                          
\begin{algorithmic}[1]
\REQUIRE $I \subseteq K[X]$, a zero-dimensional ideal .
\ENSURE $\CF = (F_1,\ldots,F_s)$, a triangular decomposition of $I \subseteq K[X]$ such that $\dimK{K[X]/I} = 
\sum_{i=1}^s \dimK{K[X]/\gen{F_i}}$.
\vspace{0.1cm}
\STATE compute $G=\{g_1,\ldots,g_m\}$, a reduced Gr\"obner basis of $I$ with respect to the 
lexicographical ordering $>$ such that $\LM(g_1)<\ldots<\LM(g_m)$;
\STATE compute $G^\prime = \{g_1^\prime,\ldots,g_{m-1}^\prime\} \subseteq K[x_2,\ldots,x_n]$ where 
$g_i^\prime$ 
is the leading coefficient of $g_i$ in $(K[x_2,\ldots,x_n])[x_1]$;
\STATE $\CF^\prime = \texttt{triangM}\left(\gen{G^\prime}\right)$;
\STATE $\CF = \{F^\prime \cup \{g_m\} \mid F^\prime \in \CF^\prime\}$;
\FOR{$1 \leq i \leq m-1$}
\IF{$g_i^\prime \notin G$}
\STATE $\CF=\CF \cup \texttt{triangM}\left(\gen{G}:g_i^\prime\right)$;
\STATE $G = G \cup \{g_i^\prime\}$;
\ENDIF
\ENDFOR
\RETURN $\CF$;
\end{algorithmic}
\end{algorithm}

Note that Algorithm \ref{algTriangDecomp} is only based on Gr\"obner basis computations, and does not
use random elements. Hence, the result is uniquely determined which allows modular computations.
In the following we fix Algorithm \ref{algTriangDecomp} to compute a triangular decomposition.

\begin{rem}
Replacing line 7 in Algorihtm \ref{algTriangDecomp} by $\CF = \CF \cup \texttt{triangM}\left(\gen 
G : g_i^{\prime\infty}
\right)$ we obtain a disjoint triangular decomposition (i.e.\ $F_i,F_j \in \CF$ with $F_i \neq F_j$ implies 
$\gen{F_i} + \gen{F_j} = K[X]$). This decomposition does in general not respect the multiplicities (i.e.\ 
$\dimK{K[X]/I} \neq \sum_{i=1}^s \dimK{K[X]/\gen{F_i}}$). 
\end{rem}

\section{Modular methods} \label{secModMeth}

One possible modular approach for solving a zero-dimensional ideal is to just replace each involved 
Gr\"obner basis computation by its corresponding modular algorithm as described in \cite{IPS11}. 
Particularly, we can replace line $1$
in Algorithm \ref{algTriangDecomp} by $G = \texttt{modStd}(I)$. In this case it is possible to apply the 
probabilistic variant since we can easily verify in the end by a simple substitution whether the solutions 
obtained from the triangular sets are really solutions of the original ideal.
Nevertheless, we propose to compute the whole triangular decomposition via modular methods actually. 
The verification is then similar by just substituting the obtained result into the input polynomials.

\vspace{0.15cm}

We consider the polynomial ring ${\mathbb{Q}}[X]$, fix a global monomial ordering $>$ on ${\mathbb{Q}}[X]$, 
and use the following notation: If $S\subseteq{\mathbb{Q}}[X]$ is a set of polynomials, then $\LM(S):=\{\LM(f)\mid 
f\in S\}$ denotes the set of leading monomials of $S$.
If $f\in{\mathbb{Q}}[X]$ is a polynomial, $I=\left\langle f_{1},\ldots,f_{r}\right\rangle \subseteq{\mathbb{Q}}[X]$ 
is an ideal, and $p$ is a prime number which does not divide any denominator of the coefficients of $f,f_{1},
\ldots,f_{r}$, then we write $f_{p} := (f\mod p) \in {\mathbb{F}}_{p}[X]$ and $I_{p}:=\left\langle (f_{1})_{p},\ldots,
(f_{r})_{p}\right\rangle \subseteq{\mathbb{F}}_{p}[X]$.

\vspace{0.15cm}

In the following, $I=\left\langle f_{1},\ldots,f_{r}\right\rangle \subseteq{\mathbb{Q}}[X]$ will be a zero-dimensional 
ideal. The triangular decomposition algorithm (Algorithm \ref{algTriangDecomp}) applied to $I$ returns a list of 
triangular sets $\CF=(F_1,\ldots,F_s)$ such that $\sqrt I = \bigcap_{i=1}^s \sqrt{\gen{F_i}}$ and $\dimQ{\Q[X]/I} = 
\sum_{i=1}^s \dimQ{\Q[X]/\gen{F_i}}$.

With respect to modularization, the following lemma obviously holds:

\begin{lem} \label{lemModularization}
With notation as above, let $p$ be a sufficiently general prime number. 
Then $I_{p}$ is zero-dimensional, $\left((F_1)_p,\ldots,(F_s)_p\right)$ is a list of triangular sets, and $\sqrt{I_p} = 
\bigcap_{i=1}^s \sqrt{\gen{(F_i)_p}}$.
\end{lem}

Relying on Lemma \ref{lemModularization}, the basic idea of the modular triangular decomposition is as follows. 
First, choose a set $\mathcal P$ of prime numbers at random. Second, compute triangular decompositions 
$\CF_p$ of $I_p$ for $p\in \mathcal P$. Third, lift the modular triangular sets to triangular sets $\CF$ over $\Q[X]$. 

The lifting process consists of two steps. First, the set $\mathcal{FP}:=\{\CF_p \mid p\in \mathcal P\}$ is lifted to
$\CF_N$ with $(F_i)_N \subseteq(\mathbb{Z}/N\mathbb{Z})[X]$ and $N:=\prod_{p\in \mathcal P}p$ by applying 
the Chinese remainder algorithm to the coefficients of the polynomials occurring in $\mathcal{FP}$. 
Second, we obtain $\CF$ with $F_i \subseteq \Q[X]$ by lifting the modular coefficients occurring in $\CF_p$ 
to rational coefficients via the Farey rational map\footnote{\emph{Farey fractions} refer to rational reconstruction. 
A definition of \emph{Farey fractions}, the \emph{Farey rational map}, and remarks on the required bound on the 
coefficients can be found in \cite{KG83}.}. This map is guaranteed to be bijective provided that $\sqrt{N/2}$ is 
larger than the moduli of all coefficients of elements in $\CF$ with $F_i \subseteq \Q[X]$. 

We now define a property of the set of primes $\mathcal{P}$ which guarantees that the lifting process is 
feasible and correct. This property is essential for the algorithm.

\begin{defn} \label{defnLucky} 
Let $\CF = (F_1,\ldots,F_s)$ be the triangular decomposition of the ideal $I$ computed by Algorithm 
\ref{algTriangDecomp}.
\begin{enumerate}
\item A prime number $p$ is called \emph{lucky} for $I$ and $\CF$ if $((F_1)_p,\ldots,(F_s)_p)$ is a
          triangular decomposition of $I_p$. Otherwise $p$ is called \emph{unlucky} for $I$ and $\CF$.
\item A set $\mathcal P$ of lucky primes for $I$ and $\CF$ is called \emph{sufficiently large} for $I$ and
          $\CF$ if $$\prod_{p\in \mathcal P} p\geq\max\{2\cdot\left\vert c\right\vert ^{2}\mid c \text{ coefficient 
          occuring in } \CF\}.$$
\end{enumerate}
\end{defn}

From a theoretical point of view, the idea of the algorithm is now as follows: 
Consider a sufficiently large set $\mathcal  P$ of lucky primes for $I$ and $\CF$, compute 
the triangular decomposition of the $I_{p}$, $p\in \mathcal P$, via Algorithm \ref{algTriangDecomp}, and lift 
the results to the triangular decomposition of $I$ as aforementioned.

From a practical point of view, we face the problem that we do not know in advance whether a prime number 
$p$ is lucky for $I$ and $\CF$.

To handle this problem, we fix a natural number $t$ and an arbitrary set of primes $\mathcal{P}$ of cardinality 
$t$. Having computed $\mathcal{FP}$, we use the following test to modify $\mathcal{P}$ such that all primes in 
$\mathcal{P}$ are lucky with high probability:

\vspace{0.2cm}

\emph{\textsc{deleteUnluckyPrimesTriang:} We define an equivalence relation on $(\mathcal{FP},\mathcal{P})$ 
by $(\CF_p,p) \sim(\CF_q,q) :\Longleftrightarrow \big(\#\CF_p = \#\CF_q \;\text{and}\; \{\LM(F_p) \mid F_p \in \CF_p\} 
= \{\LM(F_q) \mid F_q \in \CF_q\} \big).$ Then the equivalence class of largest cardinality is stored in 
$(\mathcal{FP},\mathcal{P})$, the others are deleted.}

\vspace{0.2cm}

Since we do not know a priori whether the equivalence class chosen is indeed lucky and whether it is 
sufficiently large for $I$ and $\CF$, we proceed in the following way. We lift the set $\mathcal{FP}$ to $\CF$ 
over $\Q[X]$ as described earlier, and test the result with another randomly chosen prime number:

\vspace{0.2cm}

\emph{\textsc{pTestTriang:} We randomly choose a prime number $p\notin \mathcal P$ such that $p$ does 
not divide the numerator and denominator of any coefficient occuring in a polynomial in $\{f_1,\ldots,f_r\}$ or 
$\CF$. The test returns true if $(F_p \mid F \in \CF)$ equals the triangular decomposition $\CF_p$ computed 
by the fixed Algorithm \ref{algTriangDecomp} applied on $I_p$, and false otherwise.}

\vspace{0.2cm}

If \textsc{pTestTriang} returns false, then $\mathcal{P}$ is not sufficiently large for $I$ and $\CF$ or the 
equivalence class of prime numbers chosen was unlucky. In this case, we enlarge the set $\mathcal{P}$ by 
$t$ new primes and repeat the whole process. 
On the other hand, if \textsc{pTestTriang} returns true, then we have a triangular decomposition $\CF$ of $I$ 
with high probability.
In this case, we compute the solutions $S \subseteq \C^n$ with multiplicities of the triangular sets $F = 
\{f_1,\ldots,f_n\} \in \CF$ as follows. 
Solve the univariate polynomial $f_1(x_1)$ via a univariate solver counting multiplicities, substitute $x_1$ in 
$f_2(x_1,x_2)$ by these solutions of $f_1(x_1)$, solve the corresponding univariate polynomial, and continue 
inductively this way (call this step \textsc{solveTriang}).  
Finally, we verify the result $S \subseteq \C^n$ partially by testing whether the original ideal $I \subseteq \Q[X]$ 
is contained in every $F \in \CF$, and whether the sum of the multiplicities equals the $\Q$-dimension of 
$\Q[X]/I$ (call this step \textsc{testZero}).

We summarize modular solving in Algorithm \ref{algModSolve}.\footnote{The corresponding 
procedures are implemented in \textsc{Singular} in the library \texttt{modsolve.lib}.}

\begin{algorithm}[h]
\caption{Modular Solving (\texttt{modSolve})} \label{algModSolve}
\begin{algorithmic}[1]
\REQUIRE $I \subseteq \Q[X]$, a zero-dimensional ideal.
\ENSURE $S \subseteq \C^n$, a set of points in $\C^n$ such that $f(P)=0$ for all $f \in I, P\in S$.
\vspace{0.1cm}
\STATE choose $\mathcal P$, a list of random primes; \label{numOfPrimes1}
\STATE $\mathcal{FP} = \emptyset$;
\LOOP
\FOR{$p \in \mathcal P$}
\STATE $\CF_p = \texttt{triangM}(I_p)$, the triangular decomposition of $I_p$ via Algorithm \ref{algTriangDecomp};
\STATE $\mathcal{FP} = \mathcal{FP} \cup \{\CF_p\}$;
\ENDFOR
\STATE $(\mathcal{FP},\mathcal{P}) = \textsc{deleteUnluckyPrimesTriang}(\mathcal{FP},\mathcal{P})$;
\STATE lift $(\mathcal{FP},\mathcal{P})$ to $\CF$ over $\Q[X]$ by applying Chinese remainder algorithm and 
               Farey rational map;
\IF{\textsc{pTestTriang}$(I,\CF,\mathcal{P})$}
\STATE $S = \textsc{solveTriang}(\CF)$;
\IF{$\textsc{testZero}(I,\CF,S)$}
\RETURN $S$;
\ENDIF
\ENDIF
\STATE enlarge $\mathcal P$; \label{numOfPrimes2}
\ENDLOOP
\end{algorithmic}
\end{algorithm}

\begin{rem}
In Algorithm \ref{algModSolve}, the triangular sets $\CF_p$ can be computed in parallel. Furthermore,
we can parallelize the final verification whether $I \subseteq \Q[X]$ is contained in every $F \in \CF$.
\end{rem}

\section{Examples and timings} \label{secTimings}

In this section we provide examples on which we time the algorithm \texttt{modSolve} (cf.\  Algorithm 
\ref{algModSolve}) and its parallel version as opposed to the algorithm \texttt{solve} (the procedure 
\texttt{solve} is implemented in \singular in the library \texttt{solve.lib} and computes all roots of a 
zero-dimensional input ideal using triangular sets). 
Timings are conducted by using \singular{3-1-6} on an AMD Opteron 6174 machine with $48$ CPUs, 
$2.2$ GHz, and $128$ GB of RAM running the Gentoo Linux operating system. All examples are chosen 
from The SymbolicData Project (cf.\ \cite{G13}). 

\begin{rem}
The parallelization of the modular algorithm is attained via multiple processes organized 
by \singular library code. Consequently, a future aim is to enable parallelization in the kernel 
via multiple threads.
\end{rem}

We choose the following examples to emphasize the superiority of modular solving and especially its 
parallelization:

\begin{exmp} \label{ex1}
\texttt{Cyclic\_7.xml} (cf.\ \cite{G13}).
\end{exmp}

\begin{exmp} \label{ex2}
\texttt{Verschelde\_noon6.xml} (cf.\ \cite{G13}).
\end{exmp}

\begin{exmp} \label{ex3}
\texttt{Pfister\_1.xml} (cf.\ \cite{G13}).
\end{exmp}

\begin{exmp} \label{ex4}
\texttt{Pfister\_2.xml} (cf.\ \cite{G13}).
\end{exmp}


Table \ref{tabModSolve} summarizes the results where $\texttt{modSolve}(c)$ denotes the parallelized 
version of the algorithm applied on $c$ cores. All timings are given in seconds.

\begin{table}[hbt]
\begin{center}
\begin{tabular}{|r|r|r|r|r|}
\hline
Example & \multicolumn{1}{c|}{\texttt{solve}} & \multicolumn{1}{c|}{$\texttt{modSolve}$} 
& \multicolumn{1}{c|}{$\texttt{modSolve}(10)$} & \multicolumn{1}{c|}{$\texttt{modSolve}(20)$} \\
\hline \hline
\ref{ex1} & $> 18$h & 692 & 217 & 152 \\ \hline
\ref{ex2} & 517 & 1223 & 522 & 371 \\ \hline
\ref{ex3} & 526 & 800 & 288 & 165 \\ \hline
\ref{ex4} & \hspace{1.3cm} 2250 & \hspace{1.3cm} 1276 & \hspace{1.3cm} 323 & \hspace{1.3cm} 160 \\ \hline 
\end{tabular}
\end{center}
\hspace{15mm}
\caption{Total running times in seconds for computing all roots of the considered examples via 
\texttt{solve}, \texttt{modSolve} and its parallelized variant $\texttt{modSolve}(c)$ for $c = 10,20$.} 
\label{tabModSolve}
\end{table}

\begin{rem}
Various experiments reveal that a sensitive choice of $\# \mathcal P$, the number of random primes in lines 
\ref{numOfPrimes1} and \ref{numOfPrimes2} in Algorithm \ref{algModSolve}, can decrease the running time
enormously.  To sum up, it is recommendable to relate $c$, the number of available cores, to $\# \mathcal P$.  
Particularly, in case of having more than ten cores to ones's disposal it is reasonable to set $c = \# \mathcal P$.
\end{rem}

\end{document}